\documentclass[12pt]{amsart}
\usepackage{amssymb,amsmath,amsthm}
\def\e{{\rm e}}
\def\eps{\varepsilon}

\def\d{{\rm d}}
\def\dist{{\rm dist}}

\def\R {\mathbb{R}}
\def\N {\mathbb{N}}
\def\H {{\mathcal H}}
\def\CC {J}
\def\V {{\mathcal V}}

\def\B {{\mathfrak B}}
\def\C {{\mathfrak C}}
\def\BB {{B}}
\def\D {{\rm dom}}
\def\DD {{\mathfrak D}}
\def\II {{\mathfrak I}}
\def\E {{\mathfrak E}}
\def\A {{\mathfrak A}}

\def \l {\langle}
\def \r {\rangle}

\def \and{\quad\text{and}\quad}

\def \au {\rm}
\def \ti {\it}
\def \jou {\rm}
\def \bk {\it}
\def \no#1#2#3 {{\bf #1} (#3), #2.}
\def \eds#1#2#3 {#1, #2, #3.}

\newtheorem{proposition}{Proposition}[section]
\newtheorem{theorem}[proposition]{Theorem}
\newtheorem{corollary}[proposition]{Corollary}
\newtheorem{lemma}[proposition]{Lemma}
\theoremstyle{definition}
\newtheorem{definition}[proposition]{Definition}
\newtheorem{remark}[proposition]{Remark}
\newtheorem{example}[proposition]{Example}
\numberwithin{equation}{section}

\title[Regularity of global attractors]
{On the regularity of global attractors}

\author[M. Conti, V. Pata]
{Monica Conti, Vittorino Pata}

\address{Politecnico di Milano - Dipartimento di Matematica ``F.\ Brioschi''
\newline\indent
20133 Milano, Italy}
\email{monica.conti@polimi.it}
\email{vittorino.pata@polimi.it}



\begin{document}

\begin{abstract}
This note is focused on a novel
technique in order to establish
the boundedness in more regular spaces
for global attractors
of dissipative dynamical systems,
without appealing to uniform-in-time 
estimates.
As an application of the abstract result, the semigroup
generated by the
strongly damped wave equation
$$u_{tt}-\Delta u_t-\Delta u+\varphi(u)=f$$
with critical nonlinearity is considered, 
whose attractor is shown to possess
the optimal regularity.
\end{abstract}

\maketitle

\section{Introduction}

\noindent
The evolution of many physical phenomena is
ruled by a differential equation
generating a semigroup of operators $\{S(t)\}_{t\geq0}$, 
otherwise called a dynamical system, 
acting on a suitable infinite-dimensional Banach space $\H$. In
mathematical terms,
the presence of some dissipation mechanism in the model often reflects into
the existence of an {\it absorbing set} for the semigroup. This is, by definition,
a bounded set $\B_0\subset\H$ enjoying the following
property: for any $R\geq 0$, there exists
an {\it entering time} $t_R\geq 0$
such that 
$$S(t)\B\subset\B_0,\quad\forall t\geq t_R,$$
whenever $\B\subset\H$ with $\|\B\|_\H\leq R$.
An alternative notion is the one of an {\it attracting set}, namely,
a bounded set $\C\subset\H$ satisfying the relation
$$\lim_{t\to\infty}\big[\dist_\H(S(t)\B,\C)\big]=0,$$
for all bounded sets $\B\subset\H$, where $\dist_\H$ is
the Hausdorff semidistance in $\H$, given by
$$\dist_\H(\B_1,\B_2):=\sup_{x_1\in\B_1}\inf_{x_2\in\B_2}\|x_1-x_2\|_\H.$$
Clearly, an absorbing set is attracting as well, whereas
the existence of an attracting set implies the
existence of an absorbing one. On the other hand,
an attracting set is more likely to possess nice additional 
properties,
such as compactness and finite fractal dimension.
A relevant object
providing the ultimate description of the asymptotic dynamics
is the {\it global attractor}$\,$: the unique compact set $\A\subset\H$
which is at the same time attracting
and fully invariant under the action of $S(t)$, that is,
$$S(t)\A=\A,\quad\forall t\geq 0.$$
Roughly speaking, $\A$ is the smallest possible
set where
the evolution is eventually confined. Accordingly,
any possible further regularity of the attractor
is extremely important for
a better understanding
of the longterm behavior of the semigroup.
For more details on the theory of dynamical systems
and their attractors, we address the reader
to the classical textbooks \cite{BV,HAL,HAR,TEM}
(see also the more recent~\cite{CV,CL,MZ}).

A standard way to prove the existence
of the global attractor for a (strongly continuous)
semigroup is to exhibit a compact
attracting set.
In that case, $S(t)$ is called {\it asymptotically compact},
and the attractor $\A$ turns out to be the
$\omega$-limit set of any absorbing set $\B_0$:
$$\textstyle\A=\bigcap_{t\geq 0}\overline{\bigcup_{\tau\geq t} S(\tau)\B_0}.$$
This is usually attained through the limit
\begin{equation}
\label{PIPPO}
\lim_{t\to\infty}\big[\dist_\H(S(t)\B_0,\C(t))\big]=0,
\end{equation}
where, for every fixed $t$,
$\C(t)$ is a bounded subset (say, a closed ball about the origin)
of another Banach space $\V$ compactly embedded into $\H$.
Even though \eqref{PIPPO} yields the global attractor,
no conclusion can be drawn at this stage
on the regularity of $\A$. 
However, if $\V$ is
reflexive, so that closed balls of $\V$ are closed in $\H$,
and one is able to produce the uniform-in-time estimate
\begin{equation}
\label{PLUTO}
\sup_{t\geq 0}\|\C(t)\|_\V=\varrho<\infty,
\end{equation}
it is immediate to see that $\A$ is norm-bounded in $\V$ by the 
very constant $\varrho$.

As a rule, verifying \eqref{PIPPO} in concrete situations
requires a reasonable effort; on the contrary,
showing the uniform bound~\eqref{PLUTO}, and in turn the $\V$-boundedness
of $\A$, is generally a much harder task, if not out of reach.
A paradigmatic example is the semigroup of
the damped wave equation
with a critical nonlinearity lacking monotonicity properties, 
whose global attractor
in the weak-energy space was found in \cite{ACH}, but
its optimal regularity has been obtained
only several years later~\cite{GP,ZEL}.

The aim of this note is to present a new 
and easy to handle technique apt
to establish
the boundedness of $\A$ in the higher space $\V$ 
{\it without} making use of
the uniform estimate~\eqref{PLUTO}, proving that,
to some extent, \eqref{PIPPO} alone suffices
(see Section~3).
Actually, we say more: we find a ball $\C$ of $\V$
attracting exponentially fast
all bounded subsets of $\H$; precisely,
$$\dist_\H(S(t)\B,\C)\leq \CC\big(\|\B\|_\H\big)\e^{-\omega t},$$
for some $\omega>0$
and some increasing function $\CC$.

\begin{remark}
A sufficiently regular exponentially attracting
sets is crucial in order to demonstrate the existence of
an {\it exponential attractor}$\,$:
a compact set $\E\subset\H$ of
finite fractal dimension and
positively invariant
($S(t)\E\subset\E$ for all $t\geq 0$),
which attracts bounded subsets of $\H$ at an
exponential rate,
contrary to the global attractor, whose attraction rate
can be arbitrarily slow and not measurable in
terms of the structural parameters of the problem
\cite{EFNT,EMZ,EMZ1,MZ}.
In this respect, an exponential attractor happens to be
more helpful than the global one for practical purposes,
e.g.\ numerical simulations.
\end{remark}

As an application, in the final Section~4, we consider the dynamical system
generated by the strongly damped wave equation
with a nonlinearity of critical growth, providing a simple
proof of the optimal regularity
of the related attractor.

\section{A Basic Inequality}

\noindent
We begin with some notation.
Given a Banach space $\V$ and $R\geq 0$, 
we set
$$\BB_\V(R)=\{z\in\V : \|z\|_{\V}\leq R\}.$$
We denote by $\II$ the space of continuous increasing functions $\CC:\R^+\to\R^+$,
and by $\DD$ the space of continuous decreasing functions $\beta:\R^+\to\R^+$
such that 
$\beta(\infty)<1$.

\begin{definition}
A {\it solution operator} on $\V$
is a family of
maps $U(t):\V\to \V$,
depending on $t\geq 0$,
satisfying the ``initial condition"
$$U(0)z=z,\quad\forall z\in\V.$$ 
The family $U(t)$ is called a {\it semigroup} if
it fulfills the further
property
$$U(t+\tau)=U(t)U(\tau),\quad\forall t,\tau\geq 0.$$
\end{definition}

\begin{remark}
The above maps are closely related to the study of
differential equations in Banach spaces:
suppose that, for all initial data $z\in\V$,
there is
a unique solution (in some weak sense) $\zeta:\R^+\to\V$
to the Cauchy problem
$$\begin{cases}
\frac{\d}{\d t}\zeta(t)=A(\zeta(t),t),\\
\zeta(0)=z,
\end{cases}
$$
where $A(\cdot,t)$ is a family of operators
densely defined on $\V$. Then,
we can write 
$$\zeta(t)=U(t)z.$$
When the system is autonomous, i.e.\ $A$ does
not depend explicitly on time, $U(t)$ is a semigroup.
\end{remark}

The inequality of the following lemma, although not more than
a trivial observation, is really the key idea
of the paper.

\begin{lemma}
\label{TEC}
Let $U(t)$ be a solution operator on a Banach space $\V$.
Assume that
$$\|U(t)z\|_\V\leq \beta(t)\|z\|_\V+\CC(t),\quad\forall z\in\V,
$$
for some $\beta\in\DD$, $\CC\in\II$. 
Then, for any 
$t_\star>0$ large enough so that $\beta_\star:=\beta(t_\star)<1$,
\begin{equation}
\label{UNO}
\textstyle\|U(t_\star)z\|_\V\leq \beta_\star\|z\|_\V+\frac12(1-\beta_\star)R_\star,
\end{equation}
with
$R_\star=\frac{2}{1-\beta_\star}\CC(t_\star)$.
\end{lemma}

Here is a first remarkable consequence of \eqref{UNO}.

\begin{corollary}
\label{CORABS}
If $U(t)$ is also a semigroup, then it possesses
an absorbing set.
\end{corollary}

\begin{proof}
If $\|z\|_\V\leq R_\star$, from~\eqref{UNO} and the semigroup properties
we readily get
$$\sup_{n\in\N}\|U(nt_\star)z\|_\V\leq R_\star.
$$
For an arbitrary $t\geq 0$, we write
$t=n t_\star+\tau$, with $n\in\N$
and $\tau\in[0,t_\star)$.
This yields
$$\|U(t)z\|_\V=\|U(\tau)U(nt_\star)z\|_\V 
\leq \beta(\tau)\|U(nt_\star)z\|_\V+\CC(\tau)
\leq \beta(0)R_\star+\CC(t_\star)=\kappa R_\star,
$$
having set (necessarily, $\beta(0)\geq 1$)
\begin{equation}
\label{kappa}
\textstyle \kappa=\beta(0)+\frac12(1-\beta_\star)>1.
\end{equation}
Thus, we have proved the implication
$$\|z\|_\V\leq R_\star\quad\Rightarrow
\quad \sup_{t\geq 0}\|U(t)z\|_\V\leq \kappa R_\star.
$$
Conversely, if $\|z\|_\V=R>R_\star$, we infer from \eqref{UNO} that
$$\textstyle\|U(t_\star)z\|_\V\leq \frac12(1+\beta_\star)R,
$$
and in light of the  
semigroup properties,
$$\|U(n_R t_\star)z\|_\V\leq R_\star,$$
up to taking
$$n_R=\textstyle 1
+\big\lfloor\frac{\ln R-\ln R_\star}{\ln 2-\ln(1+\beta_\star)}\big\rfloor.$$
In summary,
$$U(t)\BB_\V(R)\subset \BB_\V(\kappa R_\star),\quad\forall t\geq t_R,$$
with an entering time
$$
t_R=
\begin{cases}
0 &\text{if }R\leq R_\star,\\
n_R\,t_\star &\text{if }R> R_\star.
\end{cases}
$$
In other words, the ball $\BB_\V(\kappa R_\star)$ is an absorbing set for $U(t)$.
\end{proof}

\section{The Abstract Theorem}

\noindent
Throughout the section,
let $S(t)$ be a {\it semigroup of closed operators} 
acting on a Banach space $\H$
(cf.\ \cite{PZ}). This is a semigroup for which
the implication
$$x_n\to x,\,\,S(t)x_n\to \xi\quad\Rightarrow
\quad \xi=S(t)x
$$
holds at any fixed time $t\geq 0$, whenever $x_n,x,\xi\in\H$.

\begin{remark}
A {\it strongly continuous semigroup}, i.e.\
a semigroup enjoying the continuity
$$S(t)\in C(\H,\H),\quad\forall t\geq 0,$$
is a particular instance of a semigroup of closed operators.
\end{remark}

The semigroup $S(t)$ is also required to possess
an absorbing set $\B_0\subset\H$. Without loss of generality, 
$$\B_0=\BB_\H(R_0),\quad R_0>0.$$
Finally, let $\V$ be a reflexive Banach space compactly embedded
into $\H$.

\smallskip
Within the above assumptions,
our main result reads

\begin{theorem}
\label{MAIN}
For every $x\in\B_0$, let there exist two solution operators
$V_x(t)$ and $U_x(t)$ on $\H$
with the following properties:
\begin{itemize}
\item[(i)] For any two vectors $y,z\in\H$ satisfying $y+z=x$,
$$S(t)x=V_x(t)y+U_x(t)z.
$$
\item[(ii)] There is $\alpha\in\DD$ such that
$$\sup_{x\in\B_0}\|V_x(t)y\|_\H\leq \alpha(t)\|y\|_\H,\quad\forall y\in\B_0.
$$
\item[(iii)] There are $\beta\in\DD$ and $\CC\in\II$ such that
$$\sup_{x\in\B_0}\|U_x(t)z\|_\V\leq \beta(t)\|z\|_\V+\CC(t),\quad\forall z\in\V.
$$
\end{itemize}
Then, $\B_0$ is exponentially attracted by a closed ball
of $\V$; namely, there exist (strictly) positive constants $\varrho,K,\omega$ 
such that
\begin{equation}
\label{EAP}
\dist_\H(S(t)\B_0,\BB_\V(\varrho))\leq K\e^{-\omega t}.
\end{equation}
\end{theorem}

As a byproduct, we have

\begin{corollary}
\label{CORRI}
The semigroup $S(t)$ possesses the global attractor $\A$ bounded in $\V$.
\end{corollary}

Indeed, as mentioned in the introduction, if $S(t)$ is 
strongly continuous it is well known that~\eqref{EAP} 
implies the existence of the global attractor $\A$
subject to the bound $\|\A\|_\V\leq \varrho$
(see e.g.\ \cite{TEM}).
Same thing if $S(t)$ is only
a semigroup of closed operators (see \cite{PZ}).

\begin{proof}[Proof of Theorem \ref{MAIN}]
Let $x\in\B_0$ be arbitrarily fixed, and select 
$t_\star>0$ large enough such that
$$S(t_\star)\B_0\subset\B_0,$$
and
$$\alpha_\star:=\alpha(t_\star)<1,\quad
\beta_\star:=\beta(t_\star)<1.$$
For every $n\in\N$, we claim that the vector 
$$x_n:=S(nt_\star)x\in\B_0$$
admits the decomposition
$$
x_n=y_n+z_n,
$$
for some $y_n,z_n$ satisfying the bounds
$$\textstyle
\|y_n\|_\H\leq \alpha_\star^n R_0,\quad
\|z_n\|_\V\leq R_\star:=\frac{2}{1-\beta_\star} \CC(t_\star).
$$
We proceed by induction on $n\in\N$.
The case $n=0$ is verified by $y_0=x$, $z_0=0$.
Assume the claim true for all $n\leq m\in\N$.
Choosing
$$y_{m+1}=V_{x_m}(t_\star)y_m\quad\text{and}\quad
z_{m+1}=U_{x_m}(t_\star)z_m,
$$
we obtain the equality
$$x_{m+1}=S((m+1)t_\star)x=S(t_\star)x_m=y_{m+1}+z_{m+1}.$$
Observing that $y_m\in\B_0$, and using~\eqref{UNO},
we derive the estimates 
\begin{align*}
\|y_{m+1}\|_\H&=\|V_{x_m}(t_\star)y_m\|_\H\leq \alpha_\star
\|y_m\|_\H\leq \alpha_\star^{m+1}R_0,\\
\|z_{m+1}\|_\V&=\|U_{x_m}(t_\star)z_m\|_\V\leq
\textstyle\frac12(1+\beta_\star)R_\star\leq R_\star.
\end{align*}
This proves the claim.
Let then $t\geq 0$. 
Writing $t=n t_\star+\tau$, with $n\in\N$
and $\tau\in[0,t_\star)$,
$$S(t)x=S(\tau)x_n=V_{x_n}(\tau)y_n+U_{x_n}(\tau)z_n,$$
and
\begin{align*}
\|V_{x_n}(\tau)y_n\|_\H&\leq \alpha(0)\|y_n\|_\H
\leq \alpha(0)\alpha_\star^{-1}\alpha_\star^{t/t_\star}R_0,\\
\|U_{x_n}(\tau)z_n\|_\V&\leq\beta(0)\|z_n\|_\V+\CC(t_\star)\leq \kappa R_\star,
\end{align*}
with $\kappa>1$ as in \eqref{kappa}.
Thus, setting 
$$\varrho=\kappa R_\star,\quad K=\alpha(0)\alpha_\star^{-1}R_0,
\quad\omega=t_\star^{-1}\ln\alpha_\star^{-1},$$
the required exponential attraction property~\eqref{EAP} follows.
\end{proof}

Incidentally, Corollary~\ref{CORRI} is still true under
weaker hypotheses.

\begin{proposition}
\label{MAIN2}
Let $t_\star>0$ be
such that $S(t_\star)\B_0\subset\B_0$.
For every $x\in\B_0$, let there exist two operators
$V_{x}$ and $U_{x}$ on $\H$
with the following properties:
\begin{itemize}
\item[(i)] For any two vectors $y,z\in\H$ satisfying $y+z=x$,
$$S(t_\star)x=V_{x}y+U_{x}z.
$$
\item[(ii)] There is $\alpha_\star<1$ such that
$$\sup_{x\in\B_0}\|V_{x}y\|_\H\leq \alpha_\star\|y\|_\H,\quad\forall y\in\B_0.
$$
\item[(iii)] There are $\beta_\star<1$ and $\CC_\star\geq 0$ such that
$$\sup_{x\in\B_0}\|U_{x}z\|_\V\leq \beta_\star\|z\|_\V+\CC_\star,\quad\forall z\in\V.
$$
\end{itemize}
Then, $S(t)$ possesses the global attractor $\A$ bounded in $\V$.
\end{proposition}

\begin{proof}
Let $x\in\B_0$ be fixed.
Arguing exactly as in the proof of Theorem~\ref{MAIN},
$$
S(nt_\star)x=y_n+z_n,\quad\forall n\in\N,
$$
with 
$$
\|y_n\|_\H\leq \alpha_\star^n R_0,\quad
\textstyle \|z_n\|_\V\leq R_\star:=\frac{2}{1-\beta_\star} \CC_\star.
$$
Therefore,
$$
\dist_\H(S(nt_\star)\B_0,\BB_\V(R_\star))\leq \alpha_\star^n R_0\to 0,
$$
which is enough to establish the existence of $\A$
(cf.\ \cite{PZ}).
Since the attractor is fully invariant and contained in the absorbing set $\B_0$,
$$\A=S(nt_\star)\A\subset S(nt_\star)\B_0.$$
Hence, letting $n\to\infty$, we conclude that
$$
\dist_\H(\A,\BB_\V(R_\star))=0,
$$
yielding the set inclusion
$\A\subset\BB_\V(R_\star)$.
\end{proof}

In concrete cases, 
a commonly adopted strategy leading to
the global attractor $\A$ is finding
a decomposition
\begin{equation}
\label{DECO}
S(t)x=\eta(t;x)+\zeta(t;x),\quad\forall x\in\B_0,
\end{equation}
such that, for some function $\mu$ vanishing at infinity
and some $\CC\in\II$,
\begin{align}
\label{ALFA}
\sup_{x\in\B_0}\|\eta(t;x)\|_\H &\leq \mu(t),\\
\label{BETA}
\sup_{x\in\B_0}\|\zeta(t;x)\|_\V &\leq \CC(t).
\end{align}
However, in order to deduce 
the $\V$-boundedness of $\A$, estimate \eqref{BETA}
need be uniform in time,
same as requiring that
\begin{equation}
\label{UNIF}
\lim_{t\to\infty}\CC(t)=\varrho<\infty.
\end{equation}
Let us first dwell on a simple, albeit quite interesting, situation.

\begin{example}
\label{EX}
For two (linear and nonlinear, respectively) operators $A_0,A_1$,
assume that the differential equation
$$
\frac{\d}{\d t}\xi=A_0\xi+A_1(\xi)
$$
generates a semigroup $S(t)$ 
on $\H$.
Besides, let the linear semigroup $L(t)$, generated
by the equation with $A_1\equiv 0$,
be exponentially stable on both spaces $\H$ and $\V$, i.e.\
$$\|L(t)x\|_{\H;\V}\leq M\e^{-\delta t}\|x\|_{\H;\V},\quad\forall x\in\H;\V,$$
for some $M\geq 1$, $\delta>0$.
Finally, suppose that~\eqref{DECO}-\eqref{BETA} hold,
with 
$\eta(t;x)=L(t)x$ and
$\zeta(t;x)$ solution to
$$
\begin{cases}
\frac{\d}{\d t}\zeta=A_0\zeta+A_1(\xi),\\
\noalign{\vskip.5mm}
\zeta(0)=0,
\end{cases}
$$
where $\xi(t)=S(t)x$ (actually, \eqref{ALFA}
follows directly from exponential stability).
This kind of decomposition
has been successfully employed several times (e.g.\ \cite{GM,GT}),
and typically works for subcritical problems.
Then, setting
$$V_x(t)y=L(t)y\and
U_x(t)z=L(t)z+\zeta(t;x),$$
hypotheses (i)-(iii) of Theorem~\ref{MAIN} are easily verified.
Hence, in contrast to the standard procedure,
our approach gives at once
the $\V$-boundedness of $\A$,
with no need of \eqref{UNIF}.
\end{example}

In general, a semigroup decomposition of the form \eqref{DECO},
complying with \eqref{ALFA}-\eqref{BETA}, can be much more complicated
(cf.\ \cite{ACH,PS}).
Nonetheless, whenever 
\eqref{DECO}-\eqref{BETA} occur,
we have a strong evidence that the conclusions of Theorem~\ref{MAIN} hold true,
as in the quite challenging case of the
strongly damped wave equation with 
critical nonlinearity, discussed below.

\section{A Concrete Application}

\noindent
Consider the semilinear strongly damped wave equation
in a smooth bounded domain $\Omega\subset\R^3$
subject to Dirichlet boundary conditions
\begin{equation}
\label{SDWE}
\begin{cases}
u_{tt}-\Delta u_t-\Delta u+\varphi(u)=f,\\
u_{|\partial\Omega}=0,
\end{cases}
\end{equation}
where $f\in L^2(\Omega)$ is independent of time,
and the nonlinear term $\varphi\in C(\R)$ satisfies the
critical growth condition
\begin{equation}
\label{GROW}
|\varphi(u)-\varphi(v)|\leq c|u-v|(1+|u|^4+|v|^4),
\end{equation}
and the standard dissipativity assumption
\begin{equation}
\label{DISS}
\liminf_{|u|\to\infty}\frac{\varphi(u)}{u}>-\lambda_1.
\end{equation}
Here and in the sequel, $c$ denotes some positive constant,
while $\lambda_1>0$ is the first eigenvalue of the linear operator $A=-\Delta$ on
$L^2(\Omega)$ with
$\D(A)=H^2(\Omega)\cap H^1_0(\Omega)$.

\subsection*{Notation}
For $r\in\R$,
we introduce the scale of Hilbert spaces
(we will always omit the index $r$ when $r=0$)
$$H^r={\D}(A^{r/2}),\quad \langle u,v\rangle_{r}=\langle
A^{r/2}u,A^{r/2}v\rangle_{L^2(\Omega)},\quad
\|u\|_r=\|A^{r/2}u\|_{L^2(\Omega)},$$
and we define the product spaces
$\H^r=H^{r+1}\times H^r$.

\medskip
Equation~\eqref{SDWE} generates a strongly continuous semigroup
$S(t)$ on $\H$ possessing the global attractor $\A$
(see \cite{CC,PS}). 
In fact, for a nonlinearity of (critical) growth 
of polynomial order $5$, the existence itself of the attractor
remained an open question
for a long time.
Replacing \eqref{DISS} with the more restrictive assumption
$$\varphi \in C^1(\R),\quad\liminf_{|u|\to\infty}\varphi'(u)> -\lambda_1,$$
the boundedness of $\A$ in $\H^1$ has been demonstrated in~\cite{PZsdwe},
by means of a ``parabolic" approach.
The same paper indicates the way to obtain the result also
within~\eqref{DISS}, by means of
a rather complicated scheme borrowed from~\cite{ZEL}.
This has been recently carried out in detail
by some other authors~\cite{SCD,YS}.

Our goal is a simpler proof, which does not actually
require anything more than what already contained in \cite{PS}. 
To this end, let us first recall
some results therein.

\smallskip
\noindent
$\diamond$ $S(t)$ has an absorbing set $\B_0\subset\H$. From
now on, $c_0>0$, $\nu_0>0$ and
$\CC_0\in\II$ will denote {\it generic} constants and a {\it generic} function,
respectively, depending only
on $\B_0$.

\smallskip
\noindent
$\diamond$ The following uniform estimate holds:
$$\sup_{x\in\B_0}\sup_{t\geq 0}\|S(t)x\|_\H\leq c_0.$$

\smallskip
\noindent
$\diamond$
For every $x\in\B_0$, the solution $S(t)x=(u(t),u_t(t))$ splits into the sum
$$\hat\eta(t)+\hat\zeta(t)=(\hat v(t),\hat v_t(t))+(\hat w(t),\hat w_t(t)),$$
where
\begin{equation}
\label{CTRL}
\|\hat\eta(t)\|_\H\leq c_0\e^{-\nu_0 t}
\and
\|\hat\zeta(t)\|_{\H^{1/4}}\leq \CC_0(t).
\end{equation}

\begin{remark}
\label{REMY}
In the same spirit of~\cite{ACH},
the proofs of \cite{PS} lean on the decomposition
$$\varphi(u)=\varphi_0(u)+\varphi_1(u),$$
where the continuous functions $\varphi_0$ and $\varphi_1$
satisfy \eqref{GROW} and \eqref{DISS}, respectively,
along with
$$
\varphi_0(u)u\geq 0\quad\text{and}\quad |\varphi_1(u)|\leq c(1+|u|).
$$
This is easily obtained noting that, from \eqref{DISS}, there are
$\sigma\geq0$ and $\lambda<\lambda_1$ such that
$$|u|\geq \sigma\quad\Rightarrow
\quad \varphi(u)u\geq -\lambda u.
$$
For instance, a compatible choice is
$$
\varphi_0(u)=\gamma(u)[\varphi(u)+\lambda u],\quad \varphi_1(u)=\varphi(u)-\varphi_0(u),
$$
for any continuous
$\gamma:\R\to[0,1]$, with $\gamma(u)=0$ if $|u|\leq \sigma$ and 
$\gamma(u)=1$ if $|u|> \sigma+1$.
\end{remark}

A standard Gronwall-type lemma will also be needed.

\begin{lemma}
\label{PAZ}
Let $\Lambda:\R^+\to\R^+$ be an absolutely continuous function
satisfying
$$
\frac{\d}{\d t}\Lambda(t)+\eps\Lambda(t)
\leq k\e^{-\nu t}\Lambda(t)+\CC(t),
$$
for some $\eps,\nu,k>0$ and some $\CC\in\II$.
Then,
$$\Lambda(t)\leq \e^{k/\nu}\e^{-\eps t}\Lambda(0)+\eps^{-1}\e^{k/\nu}\CC(t).
$$
\end{lemma}

We are now in a position to state and prove

\begin{theorem}
\label{GGG}
The attractor $\A$ of the semigroup $S(t)$ on $\H$ is bounded in $\H^1$.
\end{theorem}

\begin{proof}
The first step is to apply the abstract result with $\V=\H^{1/4}$. 
To this aim, we decompose $\varphi$ as in Remark~\ref{REMY},
choosing $\sigma$ strictly positive.
Accordingly, $\varphi_0$ vanishes on the interval $[-\sigma,\sigma]$,
which allows us to write 
\begin{equation}
\label{JIM}
\varphi_0(u)=u\psi(u),\quad|\psi(u)|\leq c|u|^4.
\end{equation}
For $y,z\in\H$, we define
$$V_x(t)y=\eta(t)\and
U_x(t)z=\zeta(t),$$
where
$\eta(t)=(v(t),v_t(t))$
and $\zeta(t)=(w(t),w_t(t))$ solve the Cauchy problems
$$
\begin{cases}
v_{tt}+A v_t+A v=g,\\
\eta(0)=y,
\end{cases}
\quad
\begin{cases}
w_{tt}+A w_t+A w=h,\\
\zeta(0)=z,
\end{cases}
$$
having set
$$g=-v\psi (\hat v)\quad\text{and}
\quad h=f-\varphi(u)+v\psi(\hat v).
$$
Hypothesis (i) of Theorem~\ref{MAIN} holds
by construction, whereas
verifying (ii)-(iii) requires some passages.
By virtue of \eqref{CTRL}-\eqref{JIM}, the H\"older
inequality with exponents $(5,5/4)$
and the Sobolev embedding $H^{1}\subset L^{6}(\Omega)$,
\begin{equation}
\label{PRIMA}
\|g\|_{L^{6/5}(\Omega)}
\leq c\|v\|_1\|\hat v\|_1^4
\leq c_0\e^{-\nu_0 t}\|v\|_1.
\end{equation} 
Due to
\eqref{GROW}, \eqref{JIM} and the straightforward equality
$$
h=f-\varphi(u)+\varphi(\hat v)+\hat w\psi(\hat v)
-w\psi (\hat v)-\varphi_1(\hat v),
$$
we get
$$
|h|\leq |f|+c|\hat w|(1+|u|^4+|\hat v|^4)
+c|w||\hat v|^4+c(1+|\hat v|).
$$
Hence, making use of \eqref{CTRL}, the H\"older
inequality with exponents $(9,9/8)$
and the embeddings $H^{5/4}\subset L^{12}(\Omega)$
and $H^{1}\subset L^{6}(\Omega)$,
we obtain
\begin{align}
\label{SECONDO}
\|h\|_{L^{4/3}(\Omega)}
&\leq c\|f\|+c\|\hat w\|_{5/4}(1+\|u\|_1^4+\|\hat v\|_1^4)+c\|w\|_{5/4}\|\hat v\|_1^4
+c(1+\|\hat v\|_1)\\
\notag
&\leq c_0\e^{-\nu_0 t}\|w\|_{5/4}+\CC_0(t).
\end{align}
Then, as in \cite{PS}, we introduce the energy functionals
$$
\Lambda_0=\|\eta\|_\H^2+\eps\|v\|_1^2+2\eps\l v_t,v\r,\quad
\Lambda_1=\|\zeta\|_{\H^{1/4}}^2+\eps\|w\|_{5/4}^2+2\eps\l w_t,w\r_{1/4},
$$
with $\eps>0$ small enough in order for $\Lambda_0$ and $\Lambda_1$
to be equivalent to $\|\eta\|_\H^2$ and $\|\zeta\|_{\H^{1/4}}^2$,
respectively. From the equation for $v$ and
\eqref{PRIMA}, we infer
$$
\frac{\d}{\d t}\Lambda_0+2\eps\|v\|_1^2+2\|v_t\|_1^2-2\eps\|v_t\|^2
=2\l g,v_t+\eps v\r
\leq \|v_t\|_1^2+c_0\e^{-\nu_0 t}\|v\|_1^2.
$$
Clearly,
for $\eps$ sufficiently small,
$$
\frac{\d}{\d t}\Lambda_0+\eps\Lambda_0
\leq c_0\e^{-\nu_0 t}\Lambda_0,
$$
so that (ii) is a direct consequence of Lemma~\ref{PAZ}, which gives
$$\|V_x(t)y\|_\H^2\leq c_0\e^{-\eps t}\|y\|_\H^2.
$$
Likewise for (iii), exploiting~\eqref{SECONDO},
the H\"older
inequality with exponents $(4,4/3)$ and
the continuous embedding $H^{3/4}\subset L^{4}(\Omega)$, we
find
\begin{align*}
\frac{\d}{\d t}\Lambda_1+2\eps\|w\|_{5/4}^2+2\|w_t\|_{5/4}^2-2\eps\|w_t\|^2_{1/4}
&=2\l h,A^{1/4}w_t+\eps A^{1/4}w\r\\
&\leq \eps^2\|w\|_{5/4}^2
+\|w_t\|_{5/4}^2
+c_0\e^{-\nu_0 t}\|w\|_{5/4}^2+\CC_0(t).
\end{align*}
Therefore, taking $\eps$ small, we 
end up with the differential
inequality
$$
\frac{\d}{\d t}\Lambda_1+\eps\Lambda_1
\leq c_0\e^{-\nu_0 t}\Lambda_1+\CC_0(t),
$$
and a further application of Lemma~\ref{PAZ} provides the estimate
$$\|U_x(t)z\|_{\H^{1/4}}^2\leq c_0\e^{-\eps t}\|z\|_{\H^{1/4}}^2+\CC_0(t).
$$
By means of Theorem~\ref{MAIN}, we conclude that $\A$ is bounded in $\H^{1/4}$,
whereas the boundedness in $\H^1$ follows from
a standard bootstrap procedure. Indeed, on account of the obtained
$\H^{1/4}$-regularity,
the problem becomes in every respect {\it subcritical} for initial data
on the attractor.
In particular, $\|\varphi(u)\|$ is uniformly bounded, so that
the simple decomposition 
of Example~\ref{EX} applies, and the desired boundedness 
is drawn in one single step.
\end{proof}

As a matter of fact, Theorem~\ref{MAIN} in its full strength,
together with the transitivity
property of exponential attraction devised in~\cite{FGMZ}, yield a stronger result,
whose proof is left to the interested reader.

\begin{theorem}
\label{FFF}
There exist $\varrho>0$, $\omega>0$ and $\CC\in\II$ such that
$$
\dist_\H(S(t)\B,\BB_{\H^1}(\varrho))\leq \CC\big(\|\B\|_\H\big)\e^{-\omega t},
$$
for every bounded set $\B\subset\H$.
\end{theorem}

\begin{remark}
In light of \cite[Lemma 3.6]{PZ},
Theorem~\ref{FFF} (and so Theorem~\ref{GGG}) is easily seen to hold
replacing $\H^1$ with the more regular space $H^2\times H^2$,
provided that $\varphi\in C^1(\R)$.
\end{remark}



\end{document}